\documentclass{amsart}

\usepackage{amssymb}
\usepackage{amsmath}
\usepackage{amsthm}
\usepackage{graphics}

\allowdisplaybreaks

\numberwithin{equation}{section}

\theoremstyle{plain}
\newtheorem{thm}{Theorem}[section]

\newtheorem{lem}[thm]{Lemma}

\theoremstyle{definition}

\newtheorem{rem}[thm]{Remark}

\newcommand{\R}{\mathbb{R}}

\newcommand{\Z}{\mathbb{Z}}

\newcommand{\calF}{\mathcal{F}}
\newcommand{\calL}{\mathcal{L}}

\newcommand{\calS}{\mathcal{S}}

\begin{document}

\title[Fractional integrals on modulation spaces]
{A remark on fractional integrals on \\ modulation spaces}
\author{Mitsuru Sugimoto \and Naohito Tomita}

\address{Mitsuru Sugimoto \\
Department of Mathematics \\
Graduate School of Science \\
Osaka University \\
Toyonaka, Osaka 560-0043, Japan}
\email{sugimoto@math.sci.osaka-u.ac.jp}

\address{Naohito Tomita \\
Department of Mathematics \\
Graduate School of Science \\
Osaka University \\
Toyonaka, Osaka 560-0043, Japan}
\email{tomita@gaia.math.wani.osaka-u.ac.jp}



\maketitle

\section{introduction}
The fractional integral operator $I_{\alpha}$
is defined by
\[
I_{\alpha}f(x)
=\frac{1}{\gamma (\alpha )}\int_{\R^n}
\frac{f(y)}{|x-y|^{n-\alpha}}\, dy,
\quad \gamma (\alpha )
=\frac{\pi^{n/2}2^{\alpha}\Gamma (\alpha /2)}{\Gamma ((n-\alpha )/2)},
\]
where $0<\alpha<n$.
The well known 
Hardy-Littlewood-Sobolev theorem says that
$I_{\alpha}$ is bounded from $L^p(\R^n)$ to $L^q(\R^n)$
when $1<p<q<\infty$ and $1/q=1/p-\alpha /n$
(see \cite[Chapter 5, Theorem 1]{Stein}).
We can regard this theorem as information on how the operation
of $I_{\alpha}$ changes the
decay property of functions.
On the other hand,
the operator $I_\alpha$ can be understood
as a differential operator
of $(-\alpha)$-th order since
$\widehat{I_{\alpha}f}=|\xi|^{-\alpha}\widehat{f}$
(\cite[Chapter 5, Lemma 1]{Stein}),
and we can expect an increase in the smoothness
by acting it to functions.
\par
The purpose of this paper is to investigate
the effect of $I_\alpha$ on both decay and smoothness properties.
To study these two properties simultaneously, 
we consider the operation of $I_\alpha$
on the modulation spaces $M^{p,q}$, which were
introduced by Feichtinger \cite{Feichtinger}
(see also Triebel \cite{Triebel}).
We say that $f$ belongs to $M^{p,q}$ if its short
short-time Fourier transform
\[
V_{\varphi}f(x,\xi)
=e^{-ix\cdot\xi}[f*(M_{\xi}\varphi)](x)
=(2\pi)^{-n/2}[\widehat{f}*(M_{-x}\varphi)](\xi)
\]
is in $L^p$ (resp. $L^q$) with respect to $x$ (resp. $\xi$),
where $\varphi$ is the Gauss function $\varphi(t)=e^{-{|t|^2}/2}$.
Although the exact definition will be given in the next section,
we can see here that the decay of $V_{\varphi}f(x,\xi)$
with respect to $x$ is determined by that of $f$,
and the one with respect to $\xi$
is determined by that of $\widehat{f}$,
that is,
the smoothness of $f$.
Hence,
the first index $p$ of $M^{p,q}$ measures
the decay of $f$,
and the second index $q$ of $M^{p,q}$ measures
the smoothness of $f$.
To understand it, we remark that
$C_1(1+|t|)^{a} \le f(t) \le C_2(1+|t|)^{b}$
implies $\widetilde{C_1}(1+|t|)^{a}
\le f*\varphi(t)
\le \widetilde{C_2}(1+|t|)^{b}$,
where $a,b$ are arbitrary real numbers,
since the Gauss function is rapidly decreasing.
These explanations can be found in
Gr\"ochenig \cite[Chapter 11]{Grochenig}.
\par
Since the fractional integral operator $I_{\alpha}$
is a bounded operator
from $L^p(\R^n)$ to $L^q(\R^n)$ of convolution type,
it is easy to see that
$I_{\alpha}$ is bounded
from $M^{p_1,q_1}(\R^n)$ to $M^{p_2,q_2}(\R^n)$ when
\begin{equation}\label{(1.1)}
1/p_2=1/p_1-\alpha/n
\quad \text{and} \quad
q_1=q_2
\end{equation}
(\cite[Theorem 3.2]{Tomita}).
This boundedness says that the smoothness does not change
but the decay of $I_{\alpha}f$ is worse than that of $f$
since $M^{p_1,q_1}(\R^n)\hookrightarrow M^{p_2,q_2}(\R^n)$
in this case (see Section 2 for this embedding).
However, as we have discussed in the above,
we can expect an increase in the smoothness.
Furthermore,
since $I_{\alpha}$ is not bounded on $L^2(\R^n)$
and $M^{2,2}(\R^n)=L^2(\R^n)$,
we can easily prove that $I_{\alpha}$ is not bounded
from $M^{p_1,q_1}(\R^n)$ to $M^{p_2,q_2}(\R^n)$ when
$p_1 \ge p_2$ and $q_1 \ge q_2$
by using duality and interpolation
(see Remark \ref{4.1}).
This means that both decay and smoothness do not increase,
simultaneously.
\par
On the other hand,
Tomita \cite{Tomita} essentially proved that
$I_{\alpha}$ is bounded
from $M^{p_1,q_1}(\R^n)$ to $M^{p_2,q_2}(\R^n)$ when
\begin{equation}\label{(1.2)}
1/p_2<1/p_1-\alpha/n
\quad \text{and} \quad
1/q_2<1/q_1+\alpha/n.
\end{equation}
This boundedness says that the decay of $I_{\alpha}f$
is worse than that of $f$ by the order $\alpha/n$,
but the smoothness of $I_{\alpha}f$ is better than that of $f$
up to the order $\alpha/n$.
This result seems to be reasonable
but there still remain the problems
whether the order $\alpha/n$ is the best possible one or not and
what about the critical cases $1/p_2=1/p_1-\alpha/n$
or $1/q_2=1/q_1+\alpha/n$.
The following theorem is the complete answers to these questions:
\begin{thm}\label{1.1}
Let $0<\alpha<n$ and $1<p_1,p_2,q_1,q_2<\infty$.
Then the fractional integral operator $I_{\alpha}$
is bounded from
$M^{p_1,q_1}(\R^n)$ to $M^{p_2,q_2}(\R^n)$
if and only if
\[
1/p_2 \le 1/p_1-\alpha/n
\quad \text{and} \quad
1/q_2<1/q_1+\alpha/n.
\]
\end{thm}
Theorem \ref{1.1} says that the boundedness of $I_\alpha$ holds
even if 
$1/p_2=1/p_1-\alpha/n$,
$1/q_2<1/q_1+\alpha/n$
and $q_1>q_2$.
This is a strictly improvement of \eqref{(1.1)} and \eqref{(1.2)}.
However, the boundedness does not hold
if the second index is critical,
that is, $1/q_2=1/q_1+\alpha/n$.
We remark that \cite{Tomita} did not treat
the necessary condition for the boundedness.
\par
In order to consider the detailed behavior
of the first and second indices,
we introduce the more general operator $I_{\alpha,\beta}$ defined by
$I_{\alpha,\beta}=I_{\alpha}+I_{\beta}$,
that is,
\[
I_{\alpha,\beta}f
=\calF^{-1}\left[\left(|\xi|^{-\alpha}+|\xi|^{-\beta}
\right)\widehat{f}\right],
\]
where $0<\beta \le \alpha<n$.
We note that
$|\xi|^{-\alpha}+|\xi|^{-\beta} \sim |\xi|^{-\alpha}$
in the case $|\xi| \le 1$,
and $|\xi|^{-\alpha}+|\xi|^{-\beta} \sim |\xi|^{-\beta}$
in the case $|\xi| \ge 1$.
Since $I_{\alpha,\alpha}=2I_{\alpha}$,
we have Theorem \ref{1.1} as a corollary of the following
main result in this paper:
\begin{thm}\label{1.2}
Let $0<\beta\le\alpha<n$ and $1<p_1,p_2,q_1,q_2<\infty$.
Then $I_{\alpha,\beta}$ is bounded
from $M^{p_1,q_1}(\R^n)$ to $M^{p_2,q_2}(\R^n)$
if and only if
\[
1/p_2 \le 1/p_1-\alpha/n
\quad \text{and} \quad
1/q_2<1/q_1+\beta/n.
\]
\end{thm}
Finally we mention some related results.
Cowling, Meda and Pasquale \cite{Cowling} proved that
$I_{\alpha,\beta}$ is bounded
from $(L^{p_1}, \ell^{q_1})$ to $(L^{p_2}, \ell^{q_2})$ when
\[
1/p_2\ge 1/p_1-\beta/n
\quad \text{and} \quad
1/q_2\le 1/q_1-\alpha/n,
\]
where $(L^{p_i}, \ell^{q_i}), \, i=1,2,$ are amalgam spaces defined by
\[
\|f\|_{(L^{p}, \ell^{q})}=
\left( \sum_{k \in \Z^n}\|\varphi(\cdot-k)f\|_{L^p}^q\right)^{1/q}
\]
with an appropriate (see \eqref{(3.1)}) cut-off function $\varphi$.
The result between $I_{\alpha,\beta}$
and amalgam spaces of Lorentz type
can be also found in Cordero and Nicola \cite{Cordero}.
The definition of amalgam spaces is
based on a similar idea to that of modulation spaces
since we have the equivalence
\[
\|f\|_{M^{p,q}}\sim
\left( \sum_{k \in \Z^n}\|\calF^{-1}[\varphi(\cdot-k)\widehat{f}]
\|_{L^p}^q\right)^{1/q}.
\]
Roughly speaking, amalgam spaces are defined by a decomposition of
the function $f$ while the modulation spaces
by the same decomposition of $\widehat{f}$.
Theorem \ref{1.2} also shows a difference
between the modulation spaces and amalgam spaces,
because the boundedness of $I_{\alpha,\beta}$
on the modulation spaces
does not hold if the second index is critical.

\section{Preliminaries}
Let $\calS(\R^n)$ and $\calS'(\R^n)$ be the Schwartz spaces of
all rapidly decreasing smooth functions
and tempered distributions,
respectively.
We define the Fourier transform $\calF f$
and the inverse Fourier transform $\calF^{-1}f$
of $f \in \calS(\R^n)$ by
\[
\calF f(\xi)
=\widehat{f}(\xi)
=\int_{\R^n}e^{-i\xi \cdot x}\, f(x)\, dx
\quad \text{and} \quad
\calF^{-1}f(x)
=\frac{1}{(2\pi)^n}
\int_{\R^n}e^{ix\cdot \xi}\, f(\xi)\, d\xi.
\]
\par
We introduce the modulation spaces
based on Gr\"ochenig \cite{Grochenig}.
Fix a function $\varphi \in \calS(\R^n)\setminus \{ 0 \}$
(called the {\it window function}).
Then the short-time Fourier transform $V_{\varphi}f$ of
$f \in \calS'(\R^n)$ with respect to $\varphi$
is defined by
\[
V_{\varphi}f(x,\xi)
=(f, M_{\xi}T_x \varphi)
\qquad \text{for} \ x, \xi \in \R^n,
\]
where $M_{\xi}\varphi(t)=e^{i\xi \cdot t}\varphi(t)$,
$T_x \varphi(t)=\varphi(t-x)$
and $(\cdot,\cdot)$ denotes the inner product on $L^2(\R^n)$.
We note that,
for $f \in \calS'(\R^n)$,
$V_{\varphi}f$ is continuous on $\R^{2n}$
and $|V_{\varphi}f(x,\xi)|\le C(1+|x|+|\xi|)^N$
for some constants $C,N \ge 0$
(\cite[Theorem 11.2.3]{Grochenig}).
Let $1\le p,q \le \infty$.
Then the modulation space $M^{p,q}(\R^n)$
consists of all $f \in \calS'(\R^n)$
such that
\[
\|f\|_{M^{p,q}}
=\|V_{\varphi}f\|_{L^{p,q}}
=\left\{ \int_{\R^n} \left(
\int_{\R^n} |V_{\varphi}f(x,\xi)|^{p}\, dx
\right)^{q/p} d\xi \right\}^{1/q}
< \infty,
\]
with usual modification when $p=\infty$ or $q=\infty$.
We note that
$M^{2,2}(\R^n)=L^2(\R^n)$
(\cite[Proposition 11.3.1]{Grochenig}),
$M^{p,q}(\R^n)$ is a Banach space
(\cite[Proposition 11.3.5]{Grochenig}),
$\calS(\R^n)$ is dense in $M^{p,q}(\R^n)$
if $1 \le p,q<\infty$
(\cite[Proposition 11.3.4]{Grochenig}),
and $M^{p_1,q_1}(\R^n) \hookrightarrow M^{p_2,q_2}(\R^n)$
if $p_1 \le p_2$ and $q_1 \le q_2$
(\cite[Theorem 12.2.2]{Grochenig}).
The definition of $M^{p,q}(\R^n)$ is independent
of the choice of the window function
$\varphi \in \calS(\R^n)\setminus \{ 0 \}$,
that is,
different window functions
yield equivalent norms
(\cite[Proposition 11.3.2]{Grochenig}).
Let $\varphi \in \calS(\R^n)$
be such that
$\mathrm{supp}\, \varphi$ is compact and
$\left|\sum_{k \in \Z^n}\varphi(\xi-k)\right| \ge C>0$
for all $\xi \in \R^n$.
Then it is well known that
\begin{equation}\label{(2.1)}
\|f\|_{M^{p,q}}\sim
\left( \sum_{k \in \Z^n}\|\varphi(D-k)f\|_{L^p}^q\right)^{1/q},
\end{equation}
where $\varphi(D-k)f=\calF^{-1}[\varphi(\cdot-k)\widehat{f}]$
(see, for example, \cite{Triebel}).
The following two lemmas
will be used in the sequel.
\begin{lem}
[{\cite[Proposition, 1.3.2]{Triebel-Book}},{\cite[Lemma 3.1]{Wang}}]
\label{2.1}
Let $1 \le p \le q \le \infty$
and $\Omega \subset \R^n$ be a compact set
with $\mathrm{diam}\, \Omega<R$.
Then there exists a constant $C>0$ such that
$\|f\|_{L^q}\le C\|f\|_{L^p}$
for all $f \in \calS(\R^n)$
with $\mathrm{supp}\, \widehat{f} \subset \Omega$,
where $C$ depends only on $p,q,n$ and $R$.
In particular,
\[
\|\varphi(D-k)f\|_{L^q}
\le C\|\varphi(D-k)f\|_{L^p}
\qquad \text{for all $f \in \calS(\R^n)$ and $k \in \Z^n$},
\]
where $\varphi$ is the Schwartz function with compact support.
\end{lem}
\begin{lem}[{\cite[Chapter 4, Theorem 3]{Stein}}]\label{2.2}
Let $1<p<\infty$.
If $m \in C^{[n/2]+1}(\R^n\setminus\{0\})$ satisfies
\[
|\partial^{\gamma}m(\xi)| \le C_{\gamma}|\xi|^{-|\gamma|}
\qquad \text{for all $|\gamma| \le [n/2]+1$},
\]
then there exists a constant $C>0$ such that
\[
\|m(D)f\|_{L^p} \le C\|f\|_{L^p}
\qquad \text{for all $f \in \calS(\R^n)$},
\]
where $C$ depends only on
$p,n$ and $C_{\gamma}, \, |\gamma| \le [n/2]+1$.
\end{lem}

\section
{Sufficient condition for the boundedness of fractional integral operators}
In this section,
we prove the $\lq\lq$if" part of Theorem \ref{1.2}.
Let $\varphi \in \calS(\R^n)$
be such that
\begin{equation}\label{(3.1)}
\varphi=1 \ \text{on} \ [-1/2,1/2]^n,
\quad
\mathrm{supp}\, \varphi \subset [-3/4,3/4]^n,
\quad
\left|\sum_{k \in \Z^n}\varphi(\xi-k)\right| \ge C>0
\end{equation}
for all $\xi \in \R^n$.
\begin{lem}\label{3.1}
Let $1<p<\infty$, $\alpha \in \R$ and
\begin{equation}\label{(3.2)}
m_k^{\alpha}(\xi)=|k|^{\alpha}|\xi|^{-\alpha}\varphi(\xi-k),
\end{equation}
where $k \in \Z^n\setminus\{0\}$ and
$\varphi \in \calS(\R^n)$ is as in \eqref{(3.1)}.
Then
$\sup_{k \neq 0}\|m_k^{\alpha}(D)\|_{\calL(L^p)}<\infty$.
\end{lem}
\begin{proof}
Our proof is similar to that of \cite[Theorem 20]{Feichtinger-Narimani}.
Since
$\|m_k^{\alpha}(D)\|_{\calL(L^p)}
=\|m_k^{\alpha}(D+k)\|_{\calL(L^p)}$,
by Lemma \ref{2.2},
it is enough to show that there exists a constant $C>0$ such that
\begin{equation}\label{(3.3)}
\sup_{\xi \neq 0}
|\xi|^{|\gamma|}|\partial^{\gamma}m_k^{\alpha}(\xi+k)|
=\sup_{\xi \neq 0}|\xi|^{|\gamma|}|\partial^{\gamma}
\left(|k|^{\alpha}|\xi+k|^{-\alpha}\varphi(\xi)\right)|
\le C
\end{equation}
for all $k \neq 0$ and $|\gamma| \le [n/2]+1$.
Since $\mathrm{supp}\, \varphi \subset [-3/4,3/4]^n$,
we see that $|\xi+k| \ge 1/4$
on $\mathrm{supp}\, \varphi$ for all $k \neq 0$.
Hence,
$|k| \sim |\xi+k|$ on $\mathrm{supp}\, \varphi$
for all $k \neq 0$.
This gives \eqref{(3.3)}.
\end{proof}
We are now ready to prove the $\lq\lq$if" part of Theorem \ref{1.2}.

\medskip
\noindent
{\it Proof of $\lq\lq$if" part of Theorem \ref{1.2}.}
Let $0<\beta\le\alpha<n$,
$1<p_1,p_2,q_1,q_2<\infty$,
$1/p_2\le 1/p_1-\alpha/n$
and $1/q_2<1/q_1+\beta/n$.
We first consider the case
$1/p_2=1/p_1-\alpha/n$ and $q_1>q_2$.
In view of \eqref{(2.1)},
\begin{equation}\label{(3.4)}
\begin{split}
\|I_{\alpha,\beta}f\|_{M^{p_2,q_2}}
&\le C\left(\sum_{k \in \Z^n}
\|\varphi(D-k)(I_{\alpha,\beta}f)\|_{L^{p_2}}^{q_2}\right)^{1/{q_2}} \\
&\le \|\varphi(D)(I_{\alpha,\beta}f)\|_{L^{p_2}}
+\left(\sum_{k \neq 0}
\|\varphi(D-k)(I_{\alpha,\beta}f)\|_{L^{p_2}}^{q_2}\right)^{1/{q_2}}
\end{split}
\end{equation}
where $\varphi$ is as in \eqref{(3.1)}.
Since $0<1/p_2+\beta/n\le 1/p_2+\alpha/n=1/p_1<1$,
we can take $1<\widetilde{p_1}<\infty$
such that $1/p_2=1/\widetilde{p_1}-\beta/n$.
Note that $p_1 \le \widetilde{p_1}$.
By the Hardy-Littlewood-Sobolev theorem and Lemma \ref{2.1},
we have
\begin{equation}\label{(3.5)}
\begin{split}
\|\varphi(D)(I_{\alpha,\beta}f)\|_{L^{p_2}}
&\le \|\varphi(D)(I_{\alpha}f)\|_{L^{p_2}}
+\|\varphi(D)(I_{\beta}f)\|_{L^{p_2}} \\
&= \|I_{\alpha}(\varphi(D)f)\|_{L^{p_2}}
+\|I_{\beta}(\varphi(D)f)\|_{L^{p_2}} \\
&\le C_{\alpha}\|\varphi(D)f\|_{L^{p_1}}
+C_{\beta}\|\varphi(D)f\|_{L^{\tilde{p_1}}}
\le C\|\varphi(D)f\|_{L^{p_1}} \\
&\le C\left(\sum_{k \in \Z^n}
\|\varphi(D-k)f\|_{L^{p_1}}^{q_1}\right)^{1/{q_1}}
\le C\|f\|_{M^{p_1,q_1}}
\end{split}
\end{equation}
for all $f \in \calS(\R^n)$.
Assume that $\psi \in \calS(\R^n)$ satisfies
$\psi=1$ on $\mathrm{supp}\, \varphi$,
$\mathrm{supp}\, \psi$ is compact and
$\left|\sum_{k \in \Z^n}\psi(\xi-k)\right| \ge C>0$
for all $\xi \in \R^n$.
Then,
\begin{align*}
\varphi(D-k)(I_{\alpha,\beta}f)
&=I_{\alpha,\beta}(\varphi(D-k)f)
=I_{\alpha,\beta}(\varphi(D-k)\psi(D-k)f) \\
&=[I_{\alpha}\, \varphi(D-k)](\psi(D-k)f)
+[I_{\beta}\, \varphi(D-k)](\psi(D-k)f) \\
&=|k|^{-\alpha}m_k^{\alpha}(D)(\psi(D-k)f)
+|k|^{-\beta}m_k^{\beta}(D)(\psi(D-k)f)
\end{align*}
for all $f \in \calS(\R^n)$ and $k \neq 0$,
where $m_k^{\alpha}$ and $m_k^{\beta}$ are defined by \eqref{(3.2)}.
Hence, by Lemmas \ref{2.1} and \ref{3.1}, we have
\begin{equation}\label{(3.6)}
\begin{split}
\|\varphi(D-k)(I_{\alpha,\beta}f)\|_{L^{p_2}}
&\le C(|k|^{-\alpha}+|k|^{-\beta})
\|\psi(D-k)f\|_{L^{p_2}} \\
&\le C|k|^{-\beta}\|\psi(D-k)f\|_{L^{p_2}}
\le C|k|^{-\beta}\|\psi(D-k)f\|_{L^{p_1}}
\end{split}
\end{equation}
for all $f \in \calS(\R^n)$ and $k \neq 0$.
Set $a(k)=|k|^{-\beta}$ if $k \neq 0$, and $a(0)=1$.
Note that $\{a(k)\} \in \ell^r(\Z^n)$,
where $1/r=1/q_2-1/q_1$.
Therefore, by \eqref{(3.6)} and H\"order's inequality,
we see that
\begin{equation}\label{(3.7)}
\begin{split}
&\left( \sum_{k \neq 0}
\|\varphi(D-k)(I_{\alpha,\beta}f)\|_{L^{p_2}}^{q_2} \right)^{1/{q_2}}
\le \left\{ \sum_{k \in \Z^n}
\left( a(k)\|\psi(D-k)f\|_{L^{p_1}}\right)^{q_2} \right\}^{1/{q_2}} \\
&\le \|\{a(k)\}\|_{\ell^r}
\left(\sum_{k \in \Z^n}
\|\psi(D-k)f\|_{L^{p_1}}^{q_1} \right)^{1/{q_1}}
\le C\|f\|_{M^{p_1,q_1}}
\end{split}
\end{equation}
for all $f \in \calS(\R^n)$.
Combining \eqref{(3.4)},
\eqref{(3.5)} and \eqref{(3.7)},
we obtain the desired result with
$1/p_2=1/p_1-\alpha/n$ and $q_1>q_2$.
\par
We next consider the case
$1/p_2=1/p_1-\alpha/n$ and $q_1 \le q_2$.
Since $\beta/n>0$,
we can take $1<\widetilde{q_2}<\infty$
such that $q_1>\widetilde{q_2}$
and $1/\widetilde{q_2}<1/q_1+\beta/n$.
Note that $q_2>\widetilde{q_2}$.
Then, by the preceding case,
we see that $I_{\alpha,\beta}$ is bounded
from $M^{p_1,q_1}(\R^n)$ to $M^{p_2,\widetilde{q_2}}(\R^n)$.
This implies that
$I_{\alpha,\beta}$ is bounded
from $M^{p_1,q_1}(\R^n)$ to $M^{p_2,q_2}(\R^n)$,
since $M^{p_2,\widetilde{q_2}}(\R^n) \hookrightarrow M^{p_2,q_2}(\R^n)$.
\par
Finally,
we consider the case $1/p_2<1/p_1-\alpha/n$.
Since $0<1/p_1-\alpha/n<1$,
we can take $1<\widetilde{p_2}<\infty$ such that
$1/\widetilde{p_2}=1/p_1-\alpha/n$.
Note that $p_2>\widetilde{p_2}$.
Then, by the preceding cases,
we see that $I_{\alpha,\beta}$ is bounded
from $M^{p_1,q_1}(\R^n)$ to $M^{\widetilde{p_2},q_2}(\R^n)$.
This implies that
$I_{\alpha,\beta}$ is bounded
from $M^{p_1,q_1}(\R^n)$ to $M^{p_2,q_2}(\R^n)$,
since $M^{\widetilde{p_2},q_2}(\R^n) \hookrightarrow M^{p_2,q_2}(\R^n)$.
The proof is complete.

\section
{Necessary condition for the boundedness of fractional integral operators}
Before proving the $\lq\lq$only if" part of Theorem \ref{1.2},
we give the following remark:
\begin{rem}\label{4.1}
Let $p_1 \ge p_2$ and $q_1 \ge q_2$.
In Introduction,
we have stated that $I_{\alpha}$ is not bounded
from $M^{p_1,q_1}(\R^n)$ to $M^{p_2,q_2}(\R^n)$.
In fact,
since $M^{p_2,q_2}(\R^n) \hookrightarrow M^{p_1,q_1}(\R^n)$,
if $I_{\alpha}$ is bounded
from $M^{p_1,q_1}(\R^n)$ to $M^{p_2,q_2}(\R^n)$
then $I_{\alpha}$ is bounded on $M^{p_1,q_1}(\R^n)$.
Then, by duality,
$I_{\alpha}$ is also bounded on $M^{p_1',q_1'}(\R^n)$.
By interpolation,
the boundedness on $M^{p_1,q_1}(\R^n)$
and on $M^{p_1',q_1'}(\R^n)$ implies that
$I_{\alpha}$ is bounded on $M^{2,2}(\R^n)$.
However, since $I_{\alpha}$ is not bounded on $L^2(\R^n)$
(\cite[p.119]{Stein}),
this is a contradiction.
Hence, $I_{\alpha}$ is not bounded
from $M^{p_1,q_1}(\R^n)$ to $M^{p_2,q_2}(\R^n)$.
\end{rem}
In the rest of the paper,
we prove the $\lq\lq$only if" part of Theorem \ref{1.2}.
\begin{lem}\label{4.2}
Let $0<\beta\le\alpha<n$ and $1<p_1,p_2,q_1,q_2<\infty$.
If $I_{\alpha,\beta}$ is bounded from
$M^{p_1,q_1}(\R^n)$ to $M^{p_2,q_2}(\R^n)$,
then
$1/p_2 \le 1/p_1-\alpha/n$.
\end{lem}
\begin{proof}
We only consider the case $\alpha>\beta$,
since the proof in the case $\alpha=\beta$ is simpler.
Let $\psi \in \calS(\R^n)\setminus\{0\}$ be such that
$\mathrm{supp}\, \psi \subset [-1,1]^n$.
Set $\Psi=\calF^{-1}\psi$ and
$\Psi_{\lambda}(x)=\Psi(\lambda x)$,
where $\lambda>0$.
Then
\begin{equation}\label{(4.1)}
\varphi(D-k)\Psi_{\lambda}=
\begin{cases}
\Psi_{\lambda} &\text{if $k=0$},
\\
0 &\text{if $k \neq 0$}
\end{cases}
\end{equation}
for all $0<\lambda<1/4$,
where $\varphi$ is as in \eqref{(3.1)}.
Similarly,
\begin{equation}\label{(4.2)}
\varphi(D-k)(I_{\alpha,\beta}\Psi_{\lambda})
=I_{\alpha,\beta}(\varphi(D-k)\Psi_{\lambda})=
\begin{cases}
I_{\alpha}\Psi_{\lambda}+I_{\beta}\Psi_{\lambda} &\text{if $k=0$},
\\
0 &\text{if $k \neq 0$}
\end{cases}
\end{equation}
for all $0<\lambda<1/4$.
By \eqref{(2.1)} and \eqref{(4.1)},
we see that
\begin{equation}\label{(4.3)}
\|\Psi_{\lambda}\|_{M^{p_1,q_1}}
\le C\left( \sum_{k \in \Z^n}
\|\varphi(D-k)\Psi_{\lambda}\|_{L^{p_1}}^{q_1}\right)^{1/q_1}
=C\|\Psi_{\lambda}\|_{L^{p_1}}=C\lambda^{-n/p_1}
\end{equation}
for all $0<\lambda<1/4$.
Since $\alpha>\beta$,
we can take $0<\lambda_0<1/4$ such that
$\|I_{\alpha}\Psi\|_{L^{p_2}}\lambda_0^{-\alpha}
>2\|I_{\beta}\Psi\|_{L^{p_2}}\lambda_0^{-\beta}$.
Note that $\|I_{\alpha}\Psi\|_{L^{p_2}}\lambda^{-\alpha}
>2\|I_{\beta}\Psi\|_{L^{p_2}}\lambda^{-\beta}$
for all $0<\lambda \le \lambda_0$.
Since
$I_{\alpha}\Psi_{\lambda}(x)
=\lambda^{-\alpha}(I_{\alpha}\Psi)(\lambda x)$,
by \eqref{(2.1)} and \eqref{(4.2)},
we see that
\begin{equation}\label{(4.4)}
\begin{split}
\|I_{\alpha,\beta}\Psi_{\lambda}\|_{M^{p_2,q_2}}
&\ge C\left( \sum_{k \in \Z^n}
\|\varphi(D-k)(I_{\alpha,\beta}\Psi_{\lambda})
\|_{L^{p_2}}^{q_2}\right)^{1/q_2} \\
&=C\|I_{\alpha}\Psi_{\lambda}+I_{\beta}\Psi\|_{L^{p_2}}
\ge C
\left(\|I_{\alpha}\Psi_{\lambda}\|_{L^{p_2}}
-\|I_{\beta}\Psi_{\lambda}\|_{L^{p_2}}\right) \\
&=C\lambda^{-n/p_2}\left( 
\lambda^{-\alpha}\|I_{\alpha}\Psi\|_{L^{p_2}}
-\lambda^{-\beta}\|I_{\beta}\Psi\|_{L^{p_2}}\right) \\
&\ge C\lambda^{-n/p_2}\left( 
\lambda^{-\alpha}\|I_{\alpha}\Psi\|_{L^{p_2}}/2\right)
=C\lambda^{-n/p_2-\alpha}
\end{split}
\end{equation}
for all $0<\lambda<\lambda_0$.
Hence, by \eqref{(4.3)} and \eqref{(4.4)},
if
$I_{\alpha,\beta}$ is bounded from
$M^{p_1,q_1}(\R^n)$ to $M^{p_2,q_2}(\R^n)$,
then
\[
C_1 \lambda^{-n/p_2-\alpha}
\le \|I_{\alpha,\beta}\Psi_{\lambda}\|_{M^{p_2,q_2}}
\le \|I_{\alpha,\beta}\|_{\mathrm{op}}
\|\Psi_{\lambda}\|_{M^{p_1,q_1}} \le C_2 \lambda^{-n/p_1}
\]
for all $0<\lambda<\lambda_0$.
This implies $-n/p_2-\alpha \ge -n/p_1$,
that is,
$1/p_2 \le 1/p_1-\alpha/n$.
The proof is complete.
\end{proof}
\begin{rem}\label{4.3}
Let $0<p<\infty$ and $N$ be a sufficiently large number.
Then
\[
\begin{cases}
|x|^{-n/p}(\log|x|)^{-\alpha/p}\chi_{\{|x|>N\}} \in L^p(\R^n),
&\text{if $\alpha>1$},
\\
|x|^{-n/p}(\log|x|)^{-\alpha/p}\chi_{\{|x|>N\}} \not\in L^p(\R^n),
&\text{if $\alpha\le 1$}.
\end{cases}
\]
In fact, by a change of variables,
\[
\int_{|x|>N}
|x|^{-n}\, (\log |x|)^{-\alpha}\, dx
=C_n \int_{N}^{\infty}r^{-n}\, (\log r)^{-\alpha}\, r^{n-1}dr
=C_n \int_{\log N}^{\infty} t^{-\alpha}\, dt.
\]
\end{rem}
\begin{lem}\label{4.4}
Let $0<\beta\le\alpha<n$ and $1<p_1,p_2,q_1,q_2<\infty$.
If $1/q_2=1/q_1+\beta/n$,
then $I_{\alpha,\beta}$ is not bounded from
$M^{p_1,q_1}(\R^n)$ to $M^{p_2,q_2}(\R^n)$.
\end{lem}
\begin{proof}
We only consider the case $\alpha>\beta$,
since the proof in the case $\alpha=\beta$ is simpler.
Let $\varphi \in \calS(\R^n)$ be as in \eqref{(3.1)}.
Set
$C_{\alpha}=\sup_{k \neq 0}\|m_k^{\alpha}(D)\|_{\calL(L^{p_2})}$
and
$C_{\beta}=\sup_{k \neq 0}\|m_k^{-\beta}(D)\|_{\calL(L^{p_2})}$,
where $m_k^{\alpha}$ and $m_k^{-\beta}(D)$
are defined by \eqref{(3.2)} with $\varphi$.
Since $\alpha>\beta$,
we can take a sufficiently large natural number $N$
such that
$C_{\beta}^{-1}N^{-\beta}>2C_{\alpha}N^{-\alpha}$.
Then
\begin{equation}\label{(4.5)}
C_{\beta}^{-1}|k|^{-\beta}>2C_{\alpha}|k|^{-\alpha}
\qquad \text{for all $|k| \ge N$}.
\end{equation}
Since $1/q_2>1/q_1$,
we can take $\epsilon>0$ such that $(1+\epsilon)q_2/q_1<1$.
For these $\epsilon$ and $N$,
set
\[
f(x)=
\sum_{|\ell|>N}|\ell|^{-n/q_1}\, (\log|\ell|)^{-(1+\epsilon)/q_1}\,
e^{i\ell\cdot x}\, \Psi(x),
\]
where $\psi \in \calS(\R^n)\setminus\{0\}$ satisfies
$\mathrm{supp}\, \psi \subset [-1/4,1/4]^n$
and $\Psi=\calF^{-1}\psi$.
Since $\varphi=1$ on $[-1/2,1/2]^n$
and $\mathrm{supp}\, \varphi \subset [-3/4,3/4]^n$,
\begin{equation}\label{(4.6)}
\varphi(D-k)f(x)=
\begin{cases}
|k|^{-n/q_1}\, (\log|k|)^{-(1+\epsilon)/q_1}\,
e^{ik\cdot x}\, \Psi(x)
&\text{if $|k|>N$},
\\
0 &\text{if $|k| \le N$}.
\end{cases}
\end{equation}
Similarly,
\begin{equation}\label{(4.7)}
\varphi(D-k)I_{\alpha,\beta}f(x)=
\begin{cases}
|k|^{-n/q_1}\, (\log|k|)^{-(1+\epsilon)/q_1}\,
I_{\alpha,\beta}(M_k\Psi)(x)
&\text{if $|k|>N$},
\\
0 &\text{if $|k| \le N$},
\end{cases}
\end{equation}
where $M_k\Psi(x)=e^{ik\cdot x}\Psi(x)$.
By \eqref{(4.6)}, we have
\[
\|\varphi(D-k)f\|_{L^{p_1}}=
\begin{cases}
\|\Psi\|_{L^{p_1}}|k|^{-n/q_1}\, (\log|k|)^{-(1+\epsilon)/q_1}
&\text{if $|k|>N$},
\\
0 &\text{if $|k| \le N$}.
\end{cases}
\]
Then, by Remark \ref{4.3},
we see that $f \in M^{p_1,q_1}(\R^n)$.
On the other hand,
since
\begin{align*}
I_{\alpha}(M_k\Psi)
&=\calF^{-1}\left[ |\xi|^{-\alpha}\psi(\xi-k)\right] \\
&=|k|^{-\alpha}
\calF^{-1}\left[ \left(|k|^{\alpha}|\xi|^{-\alpha}\varphi(\xi-k)\right)
\psi(\xi-k)\right]
=|k|^{-\alpha}m_k^{\alpha}(D)(M_k\Psi)
\end{align*}
and
\begin{align*}
|k|^{-\beta}M_k\Psi
&=\calF^{-1}\left[
|k|^{-\beta}\psi(\xi-k)\right] \\
&=\calF^{-1}\left[
\left(|k|^{-\beta}|\xi|^{\beta}\varphi(\xi-k)\right)
\left(|\xi|^{-\beta}\psi(\xi-k)\right)\right]
=m_k^{-\beta}(D)I_{\beta}(M_k\Psi),
\end{align*}
by Lemma \ref{3.1},
we have
\[
\|I_{\alpha}(M_k\Psi)\|_{L^{p_2}}
\le |k|^{-\alpha}\|m_k^{\alpha}(D)\|_{\calL(L^{p_2})}
\|M_k\Psi\|_{L^{p_2}}
\le C_{\alpha}|k|^{-\alpha}\|M_k\Psi\|_{L^{p_2}}
\]
and
\[
\|M_k\Psi\|_{L^{p_2}}
\le  |k|^{\beta}\|m_k^{-\beta}(D)\|_{\calL(L^{p_2})}
\|I_{\beta}(M_k\Psi)\|_{L^{p_2}}
\le C_{\beta}|k|^{\beta}\|I_{\beta}(M_k\Psi)\|_{L^{p_2}}
\]
for all $|k|>N$.
Hence, by \eqref{(4.5)},
\begin{equation}\label{(4.8)}
\begin{split}
&\|I_{\alpha,\beta}(M_k\Psi)\|_{L^{p_2}}
\ge \|I_{\beta}(M_k\Psi)\|_{L^{p_2}}
-\|I_{\alpha}(M_k\Psi)\|_{L^{p_2}} \\
&\ge \left( C_{\beta}^{-1}|k|^{-\beta}-C_{\alpha}|k|^{-\alpha}\right)
\|M_k\Psi\|_{L^{p_2}}
\ge \left( C_{\beta}^{-1}|k|^{-\beta}/2\right)
\|\Psi\|_{L^{p_2}}
=C|k|^{-\beta}
\end{split}
\end{equation}
for all $|k|>N$.
Then,
it follows from \eqref{(4.7)} and \eqref{(4.8)} that
\begin{align*}
\|\varphi(D-k)(I_{\alpha,\beta}f)\|_{L^{p_2}}
&\ge C|k|^{-n/q_1-\beta}\, (\log|k|)^{-(1+\epsilon)/q_1} \\
&=C|k|^{-n/q_2}\, (\log|k|)^{-\{(1+\epsilon)q_2/q_1\}/q_2}
\end{align*}
for all $|k|>N$.
Also,
$\|\varphi(D-k)(I_{\alpha,\beta}f)\|_{L^{p_2}}=0$ if $|k|\le N$.
Since $(1+\epsilon)q_2/q_1<1$,
by Remark \ref{4.3}, we have
$\{|k|^{-n/q_2}(\log|k|)^{-\{(1+\epsilon)q_2/q_1\}/q_2}\}_{|k|>N}
\not\in \ell^{q_2}(\Z^n)$.
This implies
$\left(\sum_{k \in \Z^n}
\|\varphi(D-k)(I_{\alpha,\beta }f)\|_{L^{p_2}}^{q_2}\right)^{1/q_2}=\infty$,
that is, $I_{\alpha,\beta}f \not\in M^{p_2,q_2}(\R^n)$.
Therefore,
$I_{\alpha,\beta}$ is not bounded
from $M^{p_1,q_1}(\R^n)$ to $M^{p_2,q_2}(\R^n)$.
The proof is complete.
\end{proof}
We are now ready to prove the $\lq\lq$only if" part of Theorem \ref{1.2}.

\medskip
\noindent
{\it Proof of $\lq\lq$only if" part of Theorem \ref{1.2}.}
Let $0<\beta\le\alpha<n$ and $1<p_1,p_2,q_1,q_2<\infty$.
Assume that $I_{\alpha,\beta}$ is bounded
from $M^{p_1,q_1}(\R^n)$ to $M^{p_2,q_2}(\R^n)$.
Then, by Lemma \ref{4.2},
we see that $1/p_2 \le 1/p_1-\alpha/n$.
On the other hand,
if $1/q_2 \ge 1/q_1+\beta/n$
then $I_{\alpha,\beta}$ is bounded
from $M^{p_1,q_1}(\R^n)$ to $M^{p_2,\widetilde{q_2}}(\R^n)$,
since
$M^{p_2,q_2}(\R^n) \hookrightarrow M^{p_2,\widetilde{q_2}}(\R^n)$,
where $1/\widetilde{q_2}=1/q_1+\beta/n$.
However, this contradicts Lemma \ref{4.4}.
Hence, $1/q_2<1/q_1+\beta/n$.
The proof is complete.


\end{document}